\documentclass[a4paper,10pt]{amsart}
\usepackage[utf8]{inputenc}
\usepackage{amssymb, mathrsfs, amsfonts, amsmath}
\usepackage{mathtools} 
\usepackage{hyperref}
\usepackage{amsthm}
\usepackage{tikz}
\usepackage[english]{babel}
\usepackage{xcolor} 
\usepackage{geometry}
\title[Total variation of maximal functions]{Sharp total variation results for maximal functions}
\author{Jo\~ao P. G. Ramos}

\newcommand{\R}{\mathbb{R}}
\newcommand{\Z}{\mathbb{Z}}
\newcommand{\N}{\mathbb{N}}
\newcommand{\mmd}{\mathrm{d}}
\newcommand{\V}{\mathcal{V}}
\newtheorem{theorem}{Theorem}
\newtheorem{rmk}{Remark}
\newtheorem{lemm}{Lemma}
\newtheorem{prop}{Proposition}
\newtheorem{claim}{Claim} 
\newtheorem{subclaim}{Subclaim} 
\newtheorem{corol}{Corollary}
\newtheorem{proper}{Property} 

\def\Xint#1{\mathchoice
{\XXint\displaystyle\textstyle{#1}}%
{\XXint\textstyle\scriptstyle{#1}}%
{\XXint\scriptstyle\scriptscriptstyle{#1}}%
{\XXint\scriptscriptstyle\scriptscriptstyle{#1}}%
\!\int}
\def\XXint#1#2#3{{\setbox0=\hbox{$#1{#2#3}{\int}$}
\vcenter{\hbox{$#2#3$}}\kern-.5\wd0}}

\def\dashint{\Xint-}
\begin{document}
\begin{abstract} In this article, we prove some total variation inequalities for maximal functions. Our results deal with two possible generalizations of the results contained in Aldaz and Pérez Lázaro's work \cite{aldazperezlazaro}, one of whose considers a variable truncation of the maximal function, and the other one interpolates the centered and the
uncentered maximal functions. In both contexts, we find sharp constants for the desired inequalities, which can be viewed as progress towards the conjecture that the best constant for the variation inequality in the centered context is one. We also provide counterexamples showing that our methods do not apply outside the stated parameter ranges.  
\end{abstract}
\maketitle 
\section{Introduction} An object of major interest in Harmonic Analysis is the Hardy-Littlewood maximal function, 
which can be defined as 

$$ Mf(x) = \sup_{t \in \mathbb{R}_{+}} \frac{1}{2t} \int_{x-t}^{x+t} |f(s)| \mmd s.$$ 

Alternatively, one can also define its \emph{uncentered} version as $$\tilde{M}f(x) = \sup_{x \in I} \frac{1}{|I|} \int_I |f(s)| \mmd s.$$ The most classical result about these maximal functions is perhaps the Hardy--Littlewood--Wiener theorem, 
which states that both $M$ and $\tilde{M}$ map $L^p(\R)$ into itself for $1<p \le \infty$, and that in the case $p=1$ they satisfy a weak type inequality: 

$$ |\{x \in \R\colon Mf(x) > \lambda\}| \le \frac{C}{\lambda}\|f\|_1,$$
where $C = \frac{11 + \sqrt{61}}{12}$ is the best constant possible found by A. Melas \cite{Melas} for $M$. The same inequality also holds in the case of $\tilde{M}$ above, but this time with $C=2$ being the best constant, as shown by F. Riesz \cite{riesz}.\\ 

In the remarkable paper \cite{Kinnunen1}, J. Kinnunen proves, using functional analytic techniques and the aforementioned theorem, that, in fact, 
$M$ maps the Sobolev spaces $W^{1,p}(\R)$ into themselves, for $1<p\le \infty$. Kinnunen also proves that this result holds if we replace the standard maximal function by its uncentered version. 
This opened a new field of studies, and several other properties of this and other related maximal functions were studied. We mention, for example, \cite{emanuben, emanumatren, hajlaszonninen, kinnunenlindqvist, luiro}. \\ 

Since the Hardy-Littlewood maximal function fails to be in $L^1$ for every nontrivial function $f$ and the tools from functional analysis used are not available either in the case $p=1$, 
an important question was whether a bound of the form $\|(Mf)'\|_1 \le C \|f'\|_1$ could hold for every $f \in W^{1,1}$. \\

In the uncentered case, H. Tanaka \cite{Tanaka} provided us with a positive answer to this question. Explicitly, Tanaka proved that, whenever $f \in W^{1,1}(\R),$ then $\tilde{M}f$ is weakly differentiable, and it satisfies that $\|(\tilde{M}f)'\|_1 \le 2 \|f'\|_1$. Here, 
$W^{1,1}(\R)$ stands for the Sobolev space $\{f:\R \to \R\colon \|f\|_1 + \|f'\|_1 < + \infty \}.$ \\

Some years later, Aldaz and P\'erez L\'azaro \cite{aldazperezlazaro} improved Tanaka's result, showing that, whenever $f \in BV(\R)$, then the maximal function $\tilde{M}f$ is in fact 
\emph{absolutely continuous}, and $\V(\tilde{M}f) = \|(\tilde{M}f)'\|_1 \le \V(f)$, with $C=1$ being sharp, where we take the \emph{total variation of a function} to be $\V(f) := \sup_{\{x_1<\cdots<x_N\} = \mathcal{P}} \sum_{i=1}^{N-1} |f(x_{i+1}) - f(x_i)|,$ and consequently 
the space of \emph{bounded variation functions} as the space of functions $f: \R \to \R\colon \exists g; \, f= g \text{ a.e. and }\V(g) < +\infty.$ In this direction, J. Bober, E. Carneiro, K. Hughes and L. Pierce \cite{BCHP} studied the discrete version of this problem, obtaining similar results. \\

In the centered case, many questions remain unsolved. Surprisingly, it turned out to be \emph{harder} than the uncentered one, due to the contrast in smoothness of $Mf$ and $\tilde{M}f.$ 
In \cite{kurka}, O. Kurka showed the endpoint question to be true, that is, that $\V(Mf) \le C\V(f)$, with $C=240,004$. Unfortunately, his 
method does not give the best constant possible, with the standing conjecture being that $C=1$ is the sharp constant. \\ 

In \cite{temur}, F. Temur studied the discrete version of this problem, proving that 
for every $f \in BV(\mathbb{Z})$ we have $\V(Mf) \le C'\V(f),$ where $C' > 10^6$ is an absolute constant. The standing conjecture is again that $C'=1$ in this case, which was in part backed up by J. Madrid's optimal results \cite{josesito}:
If $f \in \ell^1(\mathbb{Z}),$ then $Mf \in BV(\mathbb{Z})$, and $\V(Mf) \le 2 \|f\|_1,$ with 2 being sharp in this inequality. \\ 

Our main theorems deal with -- as far as the author knows -- the first attempt to prove sharp bounded variation results for classical Hardy--Littlewood maximal functions. Indeed, we may see the classical, uncentered Hardy-Littlewood maximal function as 

\begin{align*}
\tilde{M}f(x) &= \sup_{x \in I} \frac{1}{|I|} \int_I |f(s)|\mmd s = \sup_{(y,t)\colon |x-y|\le t} \frac{1}{2t} \int_{y-t}^{y+t} |f(s)| \mmd s. \cr 
\end{align*}
Notice that this supremum need \emph{not} be attained for every function $f$ and at every point $x \in \R,$ but this shall not be a problem for us in the most diverse cases, as we will see throughout the text.
This way, we may look at this operator as a particular case of the wider class of \emph{nontangential} maximal operators 

$$ M^{\alpha}f(x) = \sup_{|x-y| \le \alpha t} \frac{1}{2t} \int_{y-t}^{y+t} |f(s)| \mmd s.$$
Indeed, from this new definition, we get directly that 
\[
\begin{cases} 
M^{\alpha}f = Mf, & \text{ if } \alpha = 0,\cr
M^{\alpha}f = \tilde{M}f, & \text{ if } \alpha =1. \cr
\end{cases} 
\]

As in the uncentered case, we can still define `truncated' versions of these operators, by imposing that $t \le R$. These operators are far from being a novelty: several references consider those all around mathematics, among those the classical \cite[Chapter~2]{stein}, and the more recent, yet related to our work, 
\cite{emanumatren}. An easy argument (see Section \ref{commrem} below) proves that, if $\alpha < \beta$, then 
\[
\V(M^{\beta}f) \le \V(M^{\alpha}f).
\]
This implies already, by the main Theorem in \cite{kurka}, that there exists a constant $A \ge 0$ such that $\V(M^{\alpha}f) \le A\V(f),$ for all $\alpha > 0.$ In the intention of sharpening this result, our first result reads, then, as follows: 

\begin{theorem}\label{angle} Fix any $f \in BV(\R).$ For every $\alpha \in [\frac{1}{3},+\infty),$ we have that 
\begin{equation}\label{maint}
\V(M^{\alpha}f) \le \V(f).
\end{equation}
There exists an extremizer $f$ for the inequality (\ref{maint}). If $\alpha > \frac{1}{3},$ then any positive extremizer $f$ to inequality (\ref{maint}) satisfies: 
\begin{itemize} 
 \item $\lim_{x \to -\infty} f(x) = \lim_{x \to + \infty} f(x).$ 
 \item There is $x_0$ such that $f$ is non-decreasing on $(-\infty,x_0)$ and non-increasing on $(x_0,+\infty).$ 
\end{itemize}
Conversely, all such functions are extremizers to our problem. Finally, for every $\alpha \ge 0$ and $f \in W^{1,1}(\R),$ $M^{\alpha}f \in W^{1,1}_{loc}(\R).$
\end{theorem}

Notice that stating that a function $g \in W^{1,1}_{loc}(\R)$ is the same as asking it to be absolutely continuous. Our ideas to prove this theorem and theorem \ref{lip} 
are heavily inspired by the ones in \cite{aldazperezlazaro}. Our aim will always be to prove that, when $f \in BV(\R),$ then the maximal function $M^{\alpha}f$ is well-behaved on the detachment set
\[
E_{\alpha} = \{x \in \R\colon M^{\alpha}f(x) > f(x)\}.
\]
Namely, we seek to obtain that the maximal function does not have any local maxima in the set where it disconnects from the original function. Such an idea, together with the concept
of the detachment set $E_{\alpha},$ are also far from being new, having already appeared at \cite{aldazperezlazaro, emanuben, emanumatren, Tanaka}, and recently at \cite{luiroradial}. More specific details of this can be found in the next section. \\ 

In general, our main ideas are contained in Lemma \ref{square}, where we prove that the region in the upper half plane that is taken into account for the supremum that defines 
$$M^1_{\equiv R}f = \sup_{x \in I\colon |I| \le 2R} \dashint_I |f(s)| \mmd s,$$
where we define 
$$\dashint_I g(s) \mmd s :=\frac{1}{|I|} \int_I g(s) \mmd s,$$ 
is actually a (rotated) \emph{square}, and not a triangle -- as a first 
glance might impress on someone --, and in the comparison of $M^{\alpha}f$ and $M^1_{\equiv R}$ over a small interval, in order to establish the maximal attachment property. \\ 

We may ask ourselves if, for instance, we could go lower than $1/3$ with this method. Our next result, however, shows that this is the optimal bound for this technique: 

\begin{theorem}\label{counterex} Let $\alpha< \frac{1}{3}$. Then there exists $f \in BV(\R)$ and a point $x_{\alpha} \in \R$ such that $x_{\alpha}$ is a local maximum of $M^{\alpha}f$, but $M^{\alpha}f(x_{\alpha}) > f(x_{\alpha}).$ 
\end{theorem} 

We can inquire ourselves whether we can generalize the results from Aldaz and P\'erez L\'azaro in yet another direction, though. With this in mind, we notice that Kurka \cite{kurka} mentions in his paper that his techniques allow one to 
prove that some Lipschitz truncations of the center maximal function, that is, maximal functions of the form 
\[
M^0_{N}f(x) = \sup_{t \le N(x)} \frac{1}{2t} \int_{x-t}^{x+t} |f(s)| \mmd s,
\]
are bounded from $BV(\R)$ to $BV(\R)$ -- with some possibly big constant -- if $\text{Lip}(N) \le 1.$ Inspired by it, we define the \emph{$N-$truncated uncentered maximal function} as 
$$ M^1_{N}f(x) = \sup_{|x-y|\le t \le N(x)} \dashint_{y-t}^{y+t} |f(s)| \mmd s. $$ 

The next result deals then with an analogous of Kurka's result in the case of the centered maximal functions. In fact, we achieve even more in this 
case, as we have also the explicit sharp constants for that. In details, the result reads as follows: 

\begin{theorem}\label{lip} Let $N: \R \to \R_{+}$ be a measurable function. If $\text{Lip}(N) \le \frac{1}{2},$ we have that, for all $f \in BV(\R),$ 
\[ 
\V(M^1_Nf) \le \V(f).
\]
Moreover, the result is sharp, in the sense that there are non-constant functions $f$ such that $\V(f) = \V(M^1_Nf).$ 
\end{theorem}

Again, we are also going to use a careful maxima in this case. Actually, we are going to do it both in theorems \ref{angle} and \ref{lip} for the non-endpoint cases $\alpha > \frac{1}{3}$ 
and $\text{Lip}(N) < \frac{1}{2},$ while the endpoints are treated with an easy limiting argument. \\ 

In the same way, one may ask whether we can ask our Lipschitz constant to be greater than $\frac{1}{2}$ in this result. Regarding this question, we prove in section 4.3 the following negative answer: 

\begin{theorem}\label{lipcont} Let $c > \frac{1}{2}$ and 
\[ 
f(x) = \begin{cases} 
        1, & \text{ if } x \in (-1,0);\cr
        0, & otherwise. \cr
       \end{cases}
\]
Then there is a function $N:\R \to \R_{\ge 0}$ such that $\text{Lip}(N) = c$ and 
\[ 
\V(M^1_Nf) = + \infty.
\]
\end{theorem}

\textbf{Acknowledgements.} The author would like to thank Christoph Thiele, for the remarks that led him to the full range $\alpha \ge \frac{1}{3}$ at Theorem \ref{angle}, as well as to the proof that this is sharp for this technique, and Olli Saari, for enlightening discussions 
about the counterexamples in the proof of Theorem \ref{lipcont} and their construction. He would also like to thank Emanuel Carneiro and Mateus Sousa for helpful comments and discussions, 
many of which took place during the author's visit to the International Centre for Theoretical Physics in Trieste, to which the author is grateful for its hospitality, and Diogo Oliveira e Silva, for his thorough inspection and numerous comments on the preliminary versions of this paper. The author would like to thank also 
the anonymous refferees, whose corrections and ideas, among which the use of a new normalization, have simplified and cleaned a lot this manuscript. Finally, the author acknowledges financial
support from the Hausdorff Center of Mathematics and the DAAD.

\section{Basic definitions and properties} Throughout the paper, $I$ and $J$ will usually denote open intervals, and $l(I),l(J),r(I),r(J)$ their left and right endpoints, respectively. 
We also denote, for $f \in BV(\R),$ the \emph{one-sided limits} $f(a+)$ and $f(a-)$ to be 
\[ 
f(a+) = \lim_{x \searrow a} f(x) \text{ and } f(a-) = \lim_{x \nearrow a} f(x).
\]
We also define, for a general function $N:\R \to \R,$ its \emph{Lipschitz constant} as 
\[
\text{Lip}(N) := \sup_{x\ne y \in \R} \frac{|N(x) - N(y)|}{|x-y|}.
\]
By considering the arguments and techniques contained in the lemmata from \cite{aldazperezlazaro}, we may consider sometimes a function in $BV(\R)$ endowed with the normalization $f(x) = \limsup_{y \to x}f(y), \,\forall x \in \R.$ At some other times, however, we might need to work with a normalization 
a little more friendly to the maximal functions involved. Let, then, for a fixed $\alpha \in (0,1],$ 

\[
\mathcal{N}_{\alpha}{f}(x) = \limsup_{(y,t) \to (x,0): |y-x|\le \alpha t} \frac{1}{2t} \int_{y-t}^{y+t} |f(s)| \mmd s.
\]
This coincides, by definition, with $f$ almost everywhere, as bounded variation functions are continuous almost everywhere. Moreover, this normalization can be stated, in a pointwise context, as 

\[ 
\mathcal{N}_{\alpha}f(x) = \frac{(1+\alpha)\limsup_{y \to x} f(y) + (1-\alpha) \liminf_{y \to x} f(y)}{2}.
\]
With this normalization, we see that, for any $f \in BV(\R),$ 
\[
M^{\alpha}f(x) \ge \mathcal{N}_{\alpha}f(x), \,\, \text{for each }x \in \R.
\]
This normalization, however, is \emph{not} friendly to boundary points: the sets $\{M^{\alpha}f > f\}$ might not be open when we adopt it, as the example of $f = \chi_{(0,\frac{1-\alpha}{4}]} + \frac{1}{2} \chi_{(\frac{1-\alpha}{4},\frac{1-\alpha}{2}]} + \chi_{(\frac{1-\alpha}{2},1]}$ endowed with $\mathcal{N}_{\alpha}f$ shows. This function has the property that $M^{\alpha}f\left(\frac{1-\alpha}{2}\right) > \mathcal{N}_{\alpha}f \left(\frac{1-\alpha}{2}\right),$ but $M^{\alpha}f = f$ at $\left(\frac{1-\alpha}{2},1\right).$\\ 

Consider then $\mathcal{N}_{\alpha}f$, and notice that the situations as in the example above can only happen if $\mathcal{N}_{\alpha}f$ is \emph{discontinuous} at a point $x$. We then let 

\begin{equation}
\tilde{\mathcal{N}}_{\alpha}f(x) = \begin{cases}
				\mathcal{N}_{\alpha}f(x), & \text{ if } M^{\alpha}f(x) > \limsup_{y \to x}f(x);\cr
				M^{\alpha}f(x), & \text{ if } \limsup_{y \to x}f(x) \ge M^{\alpha}f(x) \ge \mathcal{N}_{\alpha}f(x). \cr
				\end{cases}
\end{equation}

Of course, we are only changing the points in which $\liminf_{y \to x} f(y) < \mathcal{N}_{\alpha}f < \limsup_{y \to x} f(y),$ and thus this normalization does not change the variation, i. e., $\V(\tilde{\mathcal{N}}_{\alpha}f) = \V(f).$ 
Again, by adapting the lemmata in \cite{aldazperezlazaro} to this context, one checks that we may assume, without loss of generality, that our function has this normalization. We will, for shortness, say we are using $NORM(\alpha)$ whenever we use this normalization. Notice that $NORM(1)$ is the normalization 
used by Aldaz and P\'erez L\'azaro. \\ 

We mention also a couple of words about the maxima analysis performed throughout the paper. In the paper \cite{aldazperezlazaro}, the authors developed an ingenious 
way to prove the sharp bounded variation result for the uncentered maximal function. Namely, they proved that, whenever $f \in BV(\R)$, then the maximal function $\tilde{M}f$ is actually continuous, and the (open) set 
\[ 
E = \{ \tilde{M}f > f\} = \cup_j I_j
\]
satisfies that, in each of the intervals $I_j$, $\tilde{M}f$ has no local maxima. More specifically, they observed that every local maximum $x_0$ of $\tilde{M}f$ satisfies that $\tilde{M}f(x_0) = f(x_0).$ In our case, we are going to need the general version of this property, 
as the statement with local maxima of $M^{\alpha}f(x_0)$ may not hold. It is much more of an informal principle than a property itself, but we shall state it nonetheless, for the sake of stressing its impact on our methods. 

\begin{proper} We say that an operator $\mathcal{O}$ defined on the class of bounded variation functions has a good attachment at local maxima if, for every $f \in BV(\R)$ and local maximum $x_0$ of $\mathcal{O}f$ over an interval $(a,b)$, with $\mathcal{O}f(x_0) > \max(\mathcal{O}f(a),\mathcal{O}f(b)),$
then either $\mathcal{O}f(x_0) = |f(x_0)|$ or there exists an interval $(a,b) \supset $ such that $\mathcal{O}f$ is constant on $I$ and there is $y \in I$ such that $\mathcal{O}f(y) = |f(y)|.$ 
\end{proper} 

The intuition behind this principle is that, for such operators, one usually has that $\V(\mathcal{O}f) \le \V(f)$, as skimming through the proofs in \cite{aldazperezlazaro} suggests. This is, as one should expect, the main tool to prove Theorems \ref{angle} and \ref{lip}. 


\section{Proof of Theorems \ref{angle} and \ref{counterex}} In what follows, let $f \in BV(\R)$ have either $NORM(1)$ or $NORM(\alpha),$ where the specified normalization used will be stated in each context.

\subsection{Analysis of maxima for $M^{\alpha}$, $\alpha > \frac{1}{3}$} Here, we prove some major facts that will facilitate our work. Let then $[a,b]$ be an interval, and suppose that $M^{\alpha}f$ has a \emph{strict} local maximum at $x_0 \in (a,b).$ That is, 
we suppose that $M^{\alpha}f(x_0)$ is maximal over $[a,b],$ with $M^{\alpha}f(x_0) > \max\{ M^{\alpha}f(a),M^{\alpha}f(b)\}.$ Suppose also that
$M^{\alpha}f(x_0) = u(y,t), \text{ for some } (y,t) \in \{(z,s); |z-x_0|\le \alpha s\},$ where we define the function $u:\R\times \R_{+} \to \R_+$ as 
\[ 
u(y,t) = \frac{1}{2t} \int_{y-t}^{y+t} |f(s)|\,\mmd s. 
\]
Such an assumption is possible, as we would otherwise have that either 
\begin{itemize} 
\item a sequence $(y,t) \to (x_0,0)$ such that $\dashint_{y-t}^{y+t} |f(s)| \mmd s \to M^{\alpha}f(x_0),$ which implies $|f(x_0)| = M^{\alpha}f(x_0)$ by the normalization;
\item a sequence $(y,t)$ with $t \to \infty$ such that $\dashint_{y-t}^{y+t} |f(s)| \mmd s \to M^{\alpha}f(x_0),$ which implies that either $M^{\alpha}f(a)$ or $M^{\alpha}f(b)$ is bigger than or equal to $M^{\alpha}f(x_0),$ a contradiction. 
\end{itemize}

As $M^{\alpha}f(x_0) = u(y,t),$ we have that $M^{\alpha}f(x_0) = M^{\alpha}f(y).$ Moreover, we claim that 
\[
[y-\alpha t, y+ \alpha t] \subset (a,b).
\]
If this did not hold, then $[y-\alpha t, y+ \alpha t] \ni \text{ either } a \text{ or } b.$ Let us suppose, without loss of generality, that $a \in [y-\alpha t, y+ \alpha t].$ But then 
\[
a \ge y - \alpha t \Rightarrow |a-y| \le \alpha t \Rightarrow M^{\alpha}f(a) \ge M^{\alpha}f(y) \ge M^{\alpha}f(x_0),
\]
a contradiction to our assumption of strictness of the maximum. This implies that, as for any $z \in [y-\alpha t, y+ \alpha t] \Rightarrow |z-y|\le \alpha t,$ the maximal function $M^{\alpha}f$ is \emph{constant} over the interval $[y-\alpha t, y+ \alpha t].$ Moreover, we have that the supremum of 
\[u(z,s), \text{ for } (z,s) \in \cup_{z' \in [y-\alpha t, y+ \alpha t]} \{(z'',s''): |z''-z'| \le \alpha s''\} =: C(y,\alpha,t),
\]
is attained for $(z,s) = (y,t).$ \\ 

By standar techniques, we shall assume $f \ge 0$ from now on. Our next step is then to find a subinterval $I$ of $[y-\alpha t, y+ \alpha t]$ and a $R = R(y,\alpha,t)$ such that, over this interval $I$, it holds that 
\[ 
M^1_{\equiv R} f \equiv M^{\alpha}f. 
\]
Here, $M^1_{\equiv R}$ stands for the operator $\sup_{x \in I, |I| \le 2R} \dashint_I |f(s)| \mmd s.$ 
For that, we need to investigate a few properties of the restricted maximal function $M^1_{\equiv R} f.$ This is done via the following: 

\begin{lemm}[Boundary Projection Lemma]\label{BPL} Let $(y,t) \in \R\times \R_{+}$. Let us denote $$ \frac{1}{2t} \int_{y-t}^{y+t} f(s) \mmd s = u(y,t).$$ If $(y,t) \in \{(z,s); 0<|z-x|\le s \}$, then 

$$ u(y,t) \le \max\left\{u\left(\frac{x+y-t}{2},\frac{x-y+t}{2}\right), u\left(\frac{x+y+t}{2},\frac{y-x+t}{2}\right)\right\}.$$
 
\end{lemm}

\begin{proof}
The  proof is simple: in case $|x-y|=t,$ then the inequality is trivial, so we assume $|x-y|<t.$ We then just have to write 

\begin{align*}
u(y,t) &= \frac{1}{2t} \int_{y-t}^{y+t} f(s) \mmd s = \frac{1}{2t} \int_{y-t}^x f(s) \mmd s + \frac{1}{2t} \int_x^{y+t} f(s) \mmd s \cr 
 & = \frac{x-y+t}{2t} \frac{1}{x-y+t} \int_{y-t}^x f(s) \mmd s \cr
 & + \frac{y-x+t}{2t} \frac{1}{y-x+t} \int_x^{y+t} f(s) \mmd s \cr 
 & = \frac{x-y+t}{2t} u\left(\frac{x+y-t}{2},\frac{x-y+t}{2}\right) \cr
 & + \frac{y-x+t}{2t}  u\left(\frac{x+y+t}{2},\frac{y-x+t}{2}\right) \cr 
 & \le \max \left\{ u\left(\frac{x+y-t}{2},\frac{x-y+t}{2}\right),  u\left(\frac{x+y+t}{2},\frac{y-x+t}{2}\right)\right\}.
\end{align*}
\end{proof}

\begin{figure}
\centering
\begin{tikzpicture}[scale=0.763456]
        \draw[-,semithick] (0,-0.5) -- (-4,4.5);
        \draw[-|,semithick] (-3,-0.5) -- (0,-0.5);
	\draw[|-|,semithick] (0,-0.5) -- (2,-0.5);
	\draw[-|,semithick] (2,-0.5) -- (4,-0.5);
	\draw[|-,semithick] (4,-0.5) -- (7,-0.5);
	\draw[-,semithick] (4,-0.5) -- (8,4.5);
	\draw (2,1.1) node[circle,fill,inner sep=1pt]{} ;
	\fill [gray, opacity=0.15] (0,-0.5) -- (-4,4.5) -- (8,4.5) --(4,-0.5) -- cycle;
	\draw (2,0.6) node {$ (y,t)$};
	\draw (0,-1) node {$y-\alpha t$};
	\draw (2,-1) node {$y$};
	\draw (4,-1) node {$y+\alpha t$};
\end{tikzpicture}
\caption{The region $C(y,\alpha,t).$}
\end{figure}
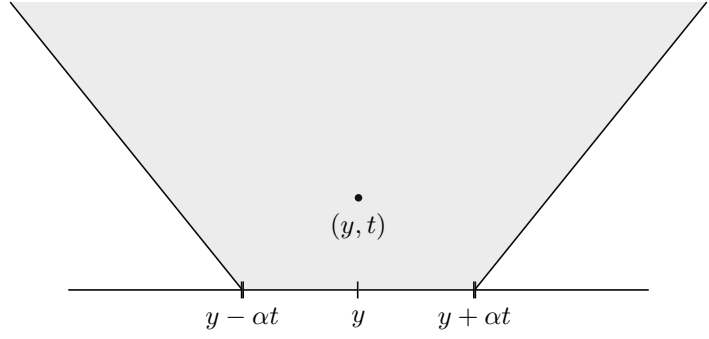
\vspace{4mm}

Let $M_{r,A}f(x) = \sup_{0 \le t \le 2A}\frac{1}{t} \int_{x}^{x+t} |f(s)| \, \mmd  s,$ and define $M_{l,A}f$ in a similar way, there the subindexes $``r"$ and $``l"$ represent, respectively, ``right" and ``left". These operators
are present out of the context of sharp regularity estimates for maximal functions, just like in \cite{riesz}. In the realm of regularity of maximal function, though, the first to introduce this notion was
Tanaka \cite{Tanaka}. As a corollary, we may obtain the following: 

\begin{corol}\label{control} For every $f \in L^1_{loc}(\R),$ it holds that
\[
\sup_{|z-x| + |t-R| \le R} u(z,t) \le \max\{M_{r,R}f(x),M_{l,R}f(x)\}.
\]
\end{corol}
From this last corollary, we are able to establish the following important -- and, as far as the author knows, new -- lemma: 

\begin{lemm}\label{square} For every $f \in L^1_{loc}(\R),$ we have also that
$$M^1_{\equiv R} f (x) = \sup_{|z-x| + |t-R| \le R} u(z,t).$$
\end{lemm}

\begin{proof}From Corollary \ref{control}, we have that 
\[
M^1_{\equiv R}f(x) := \sup_{|x-y|\le t \le R} u(y,t) \le \sup_{|z-x| + |t-R| \le R} u(z,t)
\]
\[
\le \max\{M_{r,R}f(x),M_{l,R}f(x)\} \le M^1_{\equiv R} f(x). 
\]
That is exactly what we wanted to prove. 
\end{proof}

\begin{figure} 
\begin{tikzpicture}
\draw[-|,semithick] (-4,0)--(0,0);
\draw[|-,semithick] (0,0)--(4,0);
\draw[-,thick] (0,0) -- (3,3);
\draw[-,thick] (0,0) -- (-3,3);
\draw[dashed] (1,4)--(-3/2,3/2);
\draw[dashed] (1,4)--(5/2,5/2);
\draw[-,thick] (-3,3)--(3,3);
\draw (1.5,4) node{$(y,t)$};
\draw (-3,3/2) node{$\left(\frac{x+y-t}{2},\frac{x-y+t}{2}\right)$};
\draw (4,5/2) node{$\left(\frac{x+y+t}{2},\frac{y-x+t}{2}\right)$};
\draw (0,-0.4) node{$x$};
\fill[brown, opacity=0.3] (0,0) -- (3,3) -- (-3,3) -- (0,0) -- cycle;
\end{tikzpicture} 
\caption{Illustration of Lemma \ref{BPL}: the points $\left(\frac{x+y-t}{2},\frac{x-y+t}{2}\right)$ and $\left(\frac{x+y+t}{2},\frac{y-x+t}{2}\right)$ are the projections of $(y,t)$ over the lines 
$t=x-y$ and $t=y+x,$ respectively.}
\end{figure}
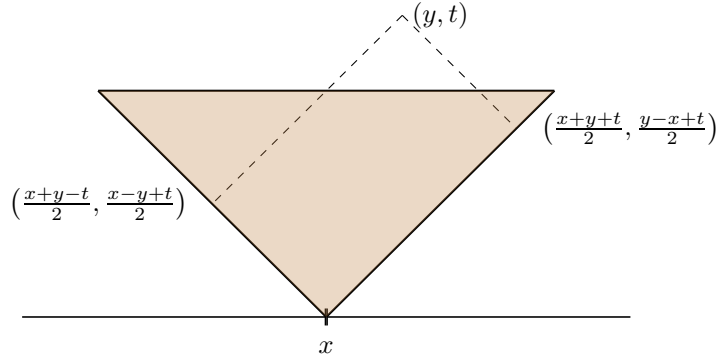

Let $R$ be then selected such that $\frac{t}{2} < R $ and $R(1-\alpha) < \alpha t.$ For $\alpha > \frac{1}{3}$ this is possible. This condition is \emph{exactly} the condition so that the region 
\[
\{(z,t'): |z-y| + |t'-R| \le R\} \subset C(y,\alpha,t).
\]
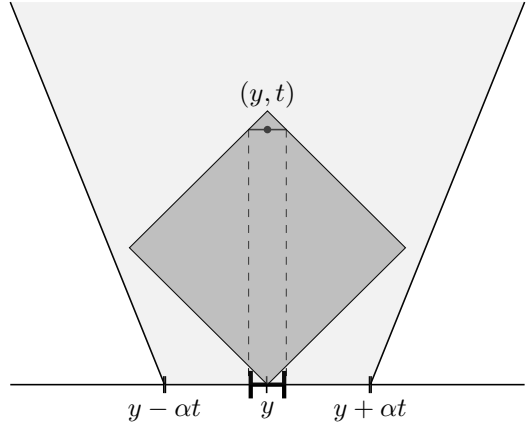
\begin{figure}
\centering
\begin{tikzpicture}[scale=0.676425]
        \draw[-,semithick] (0,-0.5) -- (-3,15/2 - 0.5);
        \draw[-|,semithick] (-3,-0.5) -- (0,-0.5);
	\draw[|-|,semithick] (0,-0.5) -- (2,-0.5);
	\draw[-|,semithick] (2,-0.5) -- (4,-0.5);
	\draw[|-,semithick] (4,-0.5) -- (7,-0.5);
	\draw[-,semithick] (4,-0.5) -- (7,15/2 - 0.5);	
	\draw (2,4.5) node[circle,fill,inner sep=1pt]{} ;
	\draw[-,semithick] (20/3-4.3,4.5) -- (8.3 - 20/3,4.5);
	\draw[dashed] (20/3-4.3,4.5) -- (20/3-4.3,-0.5);
	\draw[dashed] (8.3-20/3,4.5) -- (8.3-20/3,-0.5);
	\draw[|-|,ultra thick] (8.3-20/3,-0.5) -- (20/3-4.3,-0.5);
	\draw[draw=black, fill=gray, fill opacity=0.44523] (2,-0.5) -- (2 + 10/3 - 3.15 +2.5, 10/3 - 3.15 + 2) -- (2, 20/3 - 6.3 + 4.5) -- (2-10/3 +3.15 -2.5, 10/3 - 3.15 + 2) -- cycle;
	\fill[gray, opacity=0.1] (-3,15/2 - 0.5) -- (7,15/2 - 0.5) -- (4,-0.5) -- (0,-0.5) -- cycle;
	
	\draw (2,5.2) node {$ (y,t)$};
	\draw (0,-1) node {$y-\alpha t$};
	\draw (2,-1) node {$y$};
	\draw (4,-1) node {$y+\alpha t$};
\end{tikzpicture}\\
\caption{In the figure, the \textbf{\color{gray}dark gray} area represents the region that our Lemma gives, for some $\frac{1}{2} t < R < \frac{\alpha}{1-\alpha} t,$ and the \textbf{black} interval is one in which $M^{\alpha}f =M^1_{\equiv R} f \equiv M^{\alpha}f(y).$}
\end{figure}
Now we are able to end the proof: if $I$ is a sufficiently small interval around $y$, then, by continuity, it must hold true that the regions 
\[
\{(z,t'): |z-y'| + |t'-R| \le R\} \subset C(y,\alpha,t),
\]
for all $y' \in I.$ This is our desired interval for which $M^{\alpha}f \equiv M^1_{\equiv R} f.$ But we already know that, from \cite[Lemma~3.6]{aldazperezlazaro}, $M^1_{\equiv R} f$ satisfies a \emph{stronger} property of control of maxima. Indeed, in order to fit it into the context of Aldaz and P\'erez L\'azaro, we note that, by adopting 
$NORM(1),$ $f$ becomes automatically upper semicontinuous, and also $f \le M^1_{\equiv R}f$ everywhere. In particular, 
we know that, if $M^1_{\equiv R} f$ is constant in an interval, then it must be \emph{equal} to the function $f$ at \emph{every} point of that interval. But this is exactly our case, as we have already noticed that $M^{\alpha} f$ is constant on $[y-\alpha t, y+ \alpha t],$ 
and therefore also on $I$. This implies, in particular, that 
\[ 
M^{\alpha}f(y)= M^1_{\equiv R}f(y) = f(y),
\]
which concludes our analysis of local maxima. \\ 

\subsection{Proof of $\V(M^{\alpha}f) \le \V(f),$ for $\alpha \ge \frac{1}{3}$} We remark, before beginning, that this strategy, from now on, is essentially the same as the one contained in \cite{aldazperezlazaro}. We will, therefore, assume
that $f \ge 0$ throughout. \\ 

First, we say that a function $g : I \to \R$ is \emph{V-shaped} if there exists a point $c \in I$ such that 
\[
g|_{(l(I),c)} \text{ is non-increasing and } g|_{(c,r(I))} \text{ is non-decreasing}.
\]
We then present two different proof of this inequality, each of each suitable for a different purpose.\\

\noindent\textit{First proof: Using Lipschitz functions.} For this, we will suppose that $f$ has $NORM(1)$ as normalization. One can easily check then that $M^{\alpha}f \in C(\R)$ in this case. In fact, it is not difficult to show also that 
$M^{\alpha}f$ is continuous wt $x$ if $f$ is continuous at $x$. Moreover, we may prove an aditional property about it that will help us later: 
\begin{lemm}[Reduction to the Lipschitz case]\label{reduction} Suppose we have that  $$\V(M^{\alpha} f) \le \V(f), \;\; \forall f \in BV(\R) \cap \emph{\text{Lip}}(\R).$$ Then the same inequality holds for all Bounded Variation functions, that is, 
$$ \V(M^{\alpha} f) \le \V(f), \;\; \forall f \in BV(\R).$$ 
\end{lemm}

\begin{proof}
Let $\varphi \in \mathcal{S}(\R)$ be a smooth, nonnegative function such that $\int_{\R} \varphi (t) \mmd t = 1$, $\text{supp}(\varphi) \subset [-1,1],$ $\varphi$ is even and non-increasing on $[0,1].$ Call $\varphi_{\varepsilon} (x) = \frac{1}{\varepsilon}\varphi(\frac{x}{\varepsilon}).$ We define then $f_{\varepsilon}(x) = f * \varphi_{\varepsilon} (x).$ 
Notice that these functions are all Lipschitz (in fact, smooth) functions. Moreover, by standard theorems on Approximate Identities, we have that $f_{\varepsilon}(x) \to f(x)$ almost everywhere. Therefore, assuming the Theorem to hold for Lipschitz Functions, we have: 

\begin{align*} 
\V (M^{\alpha} f_{\varepsilon}) & \le \V(f_{\varepsilon}) \cr
 & = \sup_{x_1 < \cdots < x_N} \sum_{i=1}^{N-1} |f_{\varepsilon}(x_{i+1}) - f_{\varepsilon}(x_i)| \cr
 & \le  \int_{\R} \varphi_{\varepsilon}(t) \sup_{x_1 < \cdots < x_N} \left(\sum_{i=1}^{N-1} |f(x_{i+1} - t) - f(x_i - t)|\right)\mmd t \cr
 & \le \V(f). \cr
\end{align*}
Thus, it suffices to prove that 

\begin{equation}\label{convv}
\limsup_{y \to x} M^{\alpha}f(y) \ge \limsup_{\varepsilon \to 0} M^{\alpha}f_{\varepsilon} (x) \ge \liminf_{\varepsilon \to 0} M^{\alpha}f_{\varepsilon}(x) \ge \liminf_{y \to x} M^{\alpha}f(y), \,\,\forall x \in \R,
\end{equation}
as then
\begin{equation}\label{ineqqq}
\V(M^{\alpha}f) = \V(\liminf_{\varepsilon \to 0} M^{\alpha}f_{\varepsilon}) = \V(\limsup_{\varepsilon \to 0} M^{\alpha}f_{\varepsilon}) = \V(\lim_{j \to \infty} M^{\alpha}f_{\varepsilon_j}) \le \V(f).
\end{equation}

The proof of the equalities in \ref{ineqqq} is a direct consequence of \ref{convv}, where $\varepsilon_j \to 0$ as $j \to \infty,$ and the inequality is just a consequence of Fatou's Lemma.\\ 

Let us suppose, for the sake of a contradiction, that either the first or the third inequalities in \ref{convv} are not fulfilled. Therefore, we focus on the first inequality: suppose that there exists a real number $x_0$, a sequence $\varepsilon_k \to 0$ and a positive real number
$\eta > 0$ such that 
$$ M^{\alpha}f_{\varepsilon_k}(x_0) >(1+2\eta)\limsup_{y \to x_0}M^{\alpha}f(y).$$
By definition, there exists a sequence $(y_k,r_k)$ with 
$|y_k - x_0| \le \alpha r_k$ and 

\[ 
\dashint_{y_k-r_k}^{y_k + r_k} f_{\varepsilon_k}(s) \mmd s > (1+\eta) \limsup_{y \to x_0}M^{\alpha}f(y).
\]
\textit{Case 1:} Suppose $r_k \to 0.$ By the way we normalized $f$, there is an interval $I \ni x_0$ such that $f(y) \le (1+\eta/4)f(x_0), \forall y \in I.$ But then, by the support properties of $\varphi$ and for $k$ sufficiently large, we would have that 
$(1+\eta/2)f(x_0) \ge M^{\alpha}f_{\varepsilon_k}(x_0),$ which is a contradiction, as $\limsup_{y \to x_0}M^{\alpha}f(y) \ge f(x_0).$\\ 

\noindent\textit{Case 2:} Let then $\inf_k r_k >0.$ Then, by Fubini's theorem and manipulations, 

\begin{align*}
\dashint_{y_k-r_k}^{y_k+r_k} f_{\varepsilon_k}(s) \mmd s & = \dashint_{y_k-r_k}^{y_k+r_k} \left( \int_{-\varepsilon_k}^{\varepsilon_k} \varphi_{\varepsilon_k}(t) f(s-t) \mmd t\right) \mmd s \cr 
 & = \int_{-\varepsilon_k}^{\varepsilon_k}\varphi_{\varepsilon_k}(t) \left(\dashint_{y_k-r_k}^{y_k+r_k} f(s-t) \mmd s\right) \mmd t \cr
 & \le \frac{r_k + \varepsilon_k}{r_k} M^{\alpha}f(x_0). \cr
\end{align*}
This implies $r_k \le \frac{\varepsilon_k}{\eta} \to 0,$ which is another contradiction. \\ 

For the third inequality, we divide it once again: if $M^{\alpha}f(x_0) = u(y,t)$ for some $(y,t) \ne (x_0,0),$ then, by $L^1$ convergence of approximate identities, one easily gets that $\liminf_{\varepsilon\to 0} M^{\alpha}f_{\varepsilon} (x_0) \ge M^{\alpha}f(x_0).$ If not, pick $(y,t)$ such that $M^{\alpha}f(x_0) \le u(y,t) + \frac{\delta}{2}.$ 
Use then the $L^1$ convergenge of approximate identities in the interval $(y-t,y+t)$. The reverse inequality, and therefore the lemma, is proved, as $M^{\alpha}f(x) \ge \liminf_{y \to x} M^{\alpha}f(y).$
\end{proof}

Our main claim is then the following:

\begin{lemm}\label{lipproof}Let $f \in \text{Lip}(\R)\cap BV(\R).$ Then, over every interval of the set 
\[ 
E_{\alpha} = \{x \in \R\colon M^{\alpha}f(x) > f(x)\} = \bigcup_{j \in \Z} I_j^{\alpha},
\]
it holds that $M^{\alpha}f$ is either \emph{monotone} or \emph{V shaped} in $I_j^{\alpha}.$ 
\end{lemm} 

\begin{proof}The proof goes roughly as the first paragraph of the proof of Lemma 3.9 in \cite{aldazperezlazaro}: let $I_j^{\alpha} = (l(I_j^{\alpha}),r(I_j^{\alpha})) =: (l_j,r_j),$ and suppose that $M^{\alpha}f$ is \emph{not} V shaped there. Therefore, there would be a maximal point $x_0 \in I_j^{\alpha}$ and an interval
$J \subset I_j^{\alpha}$ such that $M^{\alpha}f$ has a \emph{strict} local maximum at $x_0$ over $J.$ Then, by the maxima analysis we performed, we see that we have reached a contradiction from this fact alone, as $J \subset E_{\alpha}.$ We omit further details, as they can be found, as already mentioned, 
at \cite[Lemma~3.9]{aldazperezlazaro}. 
\end{proof} 
We also need the following 
\begin{lemm}\label{lipat} If $f \in BV(\R) \cap \text{Lip}(\R),$ then, for every (maximal) open interval $I_j^{\alpha} \subset E_{\alpha}$, we have that 
\[ 
M^{\alpha}f(l(I_j^{\alpha})) = f(l(I_j^{\alpha})),
\]
and an analogous identity holds for $r(I_j^{\alpha}).$ 
\end{lemm}
The proof of this Lemma is straightforward, and we therefore skip it. To finalize the proof in this case for $\alpha > \frac{1}{3},$ we just notice that we can, in fact, bound the variation of $M^{\alpha}f$ \emph{inside} every interval $I_j^{\alpha}.$ In fact, we have directly from the last claim that, in case $M^{\alpha}f$ is V shaped on $I_j^{\alpha},$ then 
there exists $c_j \in I_j^{\alpha}$ such that $M^{\alpha}f$ is non-increasing on $(l_j,c_j)$ and non-decreasing on $(c_j,r_j)$. We then calculate: 
\begin{equation*} 
\begin{split} 
 \V_{I_j^{\alpha}}(M^{\alpha}f)	& = |M^{\alpha}f(l(I_j^{\alpha})) - M^{\alpha}f(c_j)| + |M^{\alpha}f(r(I_j^{\alpha})) - M^{\alpha}f(c_j)| \cr
				& \le |f(l(I_j^{\alpha})) - f(c_j)| + |f(r(I_j^{\alpha})) - f(c_j)| \cr 
				& \le V_{I_j^{\alpha}}(f).
\end{split}
\end{equation*}
The way to formally end the proof is the following: let $\mathcal{P} = \{x_1 < \cdots < x_N\}$, and let $A := \{ j \in \N: \exists x_i \in \mathcal{P}\cap I_j^{\alpha}\}.$ Clearly, the index set $A$ is finite. We refine the partition $\mathcal{P}$ by adding to it the following points: 
\begin{itemize} 
\item If $j \in A$ and $M^{\alpha}f$ is \emph{monotone} over $I_j^{\alpha},$ then add $l_j,r_j$ to the partition;
\item If $j \in A$ and $M^{\alpha}f$ is \emph{V shaped} over $I_j^{\alpha},$ then add $l_j,r_j$ and the point $c_j$ to the partition. 
\end{itemize} 
Call this new partition $\mathcal{P}'.$ By the calculation above and the fact that, if $f \in \text{Lip}(\R) \Rightarrow M^{\alpha}f \ge f \text{ everywhere},$ and in particular $M^{\alpha}f = f$ at $\R \backslash E_{\alpha},$ one obtains that 
\[
\V_{\mathcal{P}}(M^{\alpha}f) \le \V_{\mathcal{P}'}(M^{\alpha}f) \le \V(f).
\]
By taking a supremum over all partitions, we finish the result for $\alpha > \frac{1}{3}.$ On the other hand, it is straight from the definition that 
\[ 
\beta \le \alpha \Rightarrow \frac{\beta}{\alpha} M^{\alpha}f \le M^{\beta}f \le M^{\alpha}f.
\]
This implies that, for a partition $\mathcal{P}$ as above, 

\[ 
\sum_{i=1}^{N-1} |M^{\frac{1}{3}}f(x_{i+1}) - M^{\frac{1}{3}}f(x_i)| \le \lim_{\alpha \searrow \frac{1}{3}} \sum_{i=1}^{N-1} |M^{\alpha}f(x_{i+1}) - M^{\alpha}f(x_i)| \le \V(f).
\]
The theorem follows, again, as before. \\

\noindent\textit{Second proof: Directly for $f\in BV(\R)$ general.} For this part, we assume that $f$ has $NORM(\alpha)$ normalization. 
The argument here is morally the same, with just a couple of minor modifications -- and with the use of the facts we proved above, namely, that the result \emph{already} holds. Therefore, this section might seem a little bit 
superfluous now, even though its reason of being is going to be shown while we characterize the extremizers. \\ 

\begin{claim} Let $E_{\alpha} = \{x \in \R: M^{\alpha}f(x) > f(x)\}.$ This set is open for any $f \in BV(\R)$ normalized wiht $NORM(\alpha)$ and therefore can be decomposed as 
\[ 
E_{\alpha} = \cup_{j \in \Z} I_j^{\alpha},
\]
where each $I_j^{\alpha}$ is an interval. Furthermore, the restriction of $M^{\alpha}f$ to each of those intervals is either a monotone function or a V shaped function with a minimum at $c_j \in I_j^{\alpha}.$ Moreover, $M^{\alpha}f(c_j) < \min\{M^{\alpha}f(l(I_j^{\alpha})),M^{\alpha}f(r(I_j^{\alpha}))\}.$ 
\end{claim}

\begin{proof}[Proof of the claim] The claim seems quite sophisticated, but its proof is simple, once one has done the maxima analysis we have done. The fact that $E_{\alpha}$ is open is easy to see. In fact, let $x_0 \in E_{\alpha}.$ By the lower semicontinuity of $M^{\alpha}f$ at $x_0$ and the fact that we normalized
 $f$ with $NORM(\alpha),$ 
\[
\liminf_{z \to x_0} M^{\alpha}f(z) \ge M^{\alpha}f(x_0) > \limsup_{z \to x_0} f(z).
\]
This shows that, for $z$ close to $x_0$, the strict inequality should still hold, as desired. \\ 

The second part follows in the same fashion as the proof of Lemma \ref{lipproof}, and we therefore omit it. 
\end{proof}

To finish the proof of the fact that $\V_{I_j^{\alpha}}(M^{\alpha}f) \le \V_{I_j^{\alpha}}(f)$ also in this case we just need one more lemma: 
\begin{lemm}\label{interv} For every (maximal) open interval $I_j^{\alpha} \subset E_{\alpha}$ we have that 
\[ 
M^{\alpha}f(l(I_j^{\alpha})) = f(l(I_j^{\alpha})),
\]
and an analogous identity holds for $r(I_j^{\alpha}).$ 
\end{lemm}
 
This is, just like Lemma \ref{lipat}, direct from the definition and the maximality of the intervals $I_j^{\alpha}.$ The conclusion in this case uses Lemma \ref{interv} in a direct fashion, combined with the strategy for the first proof: namely, the estimate 
\begin{equation*} 
\begin{split} 
 \V_{I_j^{\alpha}}(M^{\alpha}f) & \le |M^{\alpha}f(l(I_j^{\alpha})) - M^{\alpha}f(c_j)| + |M^{\alpha}f(r(I_j^{\alpha})) - M^{\alpha}f(c_j)| \cr
				& \le |f(l(I_j^{\alpha})) - f(c_j)| + |f(r(I_j^{\alpha})) - f(c_j)| \cr 
				& \le V_{I_j^{\alpha}}(f)
\end{split}
\end{equation*}
still holds, by Lemma \ref{interv} and by the fact that $c_j \in I_j^{\alpha}.$ This finishes finally the second proof of Theorem \ref{angle}.

\subsection{Absolute continuity on the detachment set} We prove briefly the fact that, for $f \in W^{1,1}(\R),$ then we have that $M^{\alpha}f \in W^{1,1}_{loc}(\R)$ for
\emph{any} $1>\alpha>0,$ as the case $\alpha = 0$ has been dealt with by Kurka \cite{kurka}, in Corollary 1.4. \\ 

Indeed, let 
\[
E_{\alpha,k} = \{ x \in E_{\alpha} \colon  M^{\alpha}f(x) = \sup_{(y,t) \colon |y-x|\le \alpha t, \,t \ge \frac{1}{2k}} \frac{1}{2t}\int_{y-t}^{y+t} |f(s)| \mmd s \}. 
\]
Then we see that $E_{\alpha}= \cup_{k \ge 1} E_{\alpha,k}.$ Moreover, for $x,y \in E_{\alpha,k},$ let then $(y_1,t_1)$ have this property for $x$. Suppose also, without loss of generality, that $y \ge x$ and $M^{\alpha}f(x) > M^{\alpha}f(y).$ By assuming that $y > y_1 + \alpha t_1$ -- as otherwise $M^{\alpha}f(x) \le M^{\alpha}f(y)$ --, we have that 
\begin{equation*}
\begin{split}
M^{\alpha}f(x) - M^{\alpha}f(y) & \le \frac{1}{2t_1} \int_{y_1-t_1}^{y_1 + t_1} |f(s)| \mmd s -u\left(\frac{y+\alpha y_1 - \alpha t_1}{1+\alpha},\frac{y-y_1+t_1}{1+\alpha}\right) \cr 
				& \le \frac{\frac{2}{1+\alpha}(y-y_1) -\frac{2\alpha}{1+\alpha}t_1}{2t_1 \cdot \frac{2}{1+\alpha}(y-y_1+t_1)} \int_{y_1 - t_1}^{y_1 + t_1} |f(s)| \mmd s \cr 
				& \le \frac{\frac{2}{1+\alpha}|y-x|}{\frac{2}{1+\alpha}(y-y_1+t_1)} \|f\|_{\infty}  \le \frac{|x-y|}{2t_1} \|f\|_{\infty}  \le k|x-y|\|f\|_{\infty}.\cr 
\end{split}
\end{equation*} 
This shows that $M^{\alpha}f$ is Lipschitz continuous with constant $\le k\|f\|_{\infty}$ on each $E_{\alpha,k}.$ The proof of the asserted fact, however, follows from this, by using the well-known Banach-Zarecki 
lemma: 
\begin{lemm}[Banach-Zarecki]\label{BZ} A function $g: I \to \R$ is absolutely continuous if and only if the following conditions hold simultaneously: 
\begin{enumerate} 
\item $g$ is continuous; 
\item $g$ is of bounded variation;
\item $g(S)$ has measure zero for every set $S \subset I$ with $|S| = 0.$ 
\end{enumerate}
\end{lemm} 
In fact, let $S$ be then a null-measure set on the real line and $f \in W^{1,1}(\R)$ -- which implies that $M^{\alpha}f \in C(\R)$ --, and let us invoke \cite[Lemma~3.1]{aldazperezlazaro}:

\begin{lemm}\label{lipk} Let $f: I \to \R$ be a continuous function. Let also $E \subset \{x \in I \colon |\overline{D}f(x)|:=\left| \limsup_{h \to 0} \frac{f(x+h)-f(x)}{h}\right| \le k\}.$ Then 
\[
m^{*}(f(E)) \le k m^{*}(E),
\]
where $m(S) = |S|$ stands for the Lebesgue measure of $S$.
\end{lemm} 
It is easy to see that the maximal functions $M^{\alpha}f$ are, in fact, \emph{continuous} on the open set $E_{\alpha}.$ Thus, we may use Lemmas \ref{BZ} and \ref{lipk} in each of the connected components of $E_{\alpha}:$
\begin{equation*} 
\begin{split}
|M^{\alpha}f(S\cap I_j^{\alpha})| &\le \sum_{k \ge 1} |M^{\alpha}f(S \cap E_{\alpha,k} \cap I_j^{\alpha})| = 0, \cr
\end{split}
\end{equation*} 
where we used that $M^{\alpha}f$ is Lipschitz over each $E_{\alpha,k}.$ But this implies that 
\begin{align*}
|M^{\alpha}f(S)| & \le |M^{\alpha}f(S\cap E_{\alpha}^c)| + \sum_{j \in \Z} |M^{\alpha}f(S \cap I_j^{\alpha})|\cr
		 & = |f(S \cap E_{\alpha}^c)| = 0,
\end{align*}
by Lemma \ref{BZ} and the fact that $f \in W^{1,1}_{loc}(\R).$ This finishes this part of the analysis. 

\subsection{Sharpness of the inequality and extremizers} In this part, we prove that the best constant in such inequalities is indeed 1, and characterize the extremizers for such. Namely, we mention promptly that the inequality 
\emph{must} be sharp, as $f = \chi_{(-1,0)}$ realizes equality. \\ 
It is easy to see that, to do so, we may assume that $f$ still has $NORM(\alpha)$ normalization. \\

\begin{claim}\label{inteq} Let $f \in BV(\R)$ normalized as before satisfy $\V(f) = \V(M^{\alpha}f).$ If we decompose $E_{\alpha} = \cup_{j} I_j^{\alpha},$ where each of the $I_j^{\alpha}$ is open and maximal, then
\[ 
\V_{I_j^{\alpha}}(f) = \V_{I_j^{\alpha}}(M^{\alpha}f).
\]
\end{claim} 

\begin{proof} Let $\mathcal{P}, \mathcal{Q}$ be two finite partitions of $\R$ such that 
\begin{equation}\label{partt}
\begin{cases} 
 \V(M^{\alpha}f) & \le \V_{\mathcal{P}} (M^{\alpha}f) + \frac{\varepsilon}{20},\cr 
 \V(f) & \le \V_{\mathcal{Q}}(f) + \frac{\varepsilon}{20}. \cr
\end{cases}
\end{equation}
Now let the mutual refinement of those be $\mathcal{S} = \mathcal{P} \cup \mathcal{Q}.$ We consider the intersection $\mathcal{S} \cap E_{\alpha}:$ if the finite set $A:= \{ i\colon I_j^{\alpha} \cap \mathcal{S} \ne \emptyset\}$ satisfies that
\begin{equation}\label{finvar} \sum_{j \in A} \V_{I_j^{\alpha}} (f) \ge \sum_{j \in \N}\V_{I_j^{\alpha}}(f) - \frac{\varepsilon}{20},
\end{equation}
then keep the partition as it is before advancing. If not, then add to $\mathcal{S}$ finitely many points, all of them contained in intervals of the form $\overline{I_j^{\alpha}}$, such that inequality \ref{finvar} holds. Call this
new partition $\mathcal{S}$ again, as it still satisfies the inequalities \ref{partt}. \\

We finally add some other points to the partition $\mathcal{S}:$ If $j \not\in A,$ do not add any points from the interval. If $j \in A,$ then do the following: 
\begin{enumerate}
 \item As $f = M^{\alpha}f$ on the boundary of an interval $I_j^{\alpha}$, we add to the collection both endpoints $r(I_j^{\alpha}),l(I_j^{\alpha}).$  
 \item If $M^{\alpha}f$ is V shaped over the interval $I_j^{\alpha}$, then there is a point $c_j$ such that $M^{\alpha}f$ is non-increasing on $(l_j,c_j)$ and non-decreasing on $(c_j,r_j).$  Add such a point to our partition.
 \item If $\V_{I_j^{\alpha}}(f) > \V_{ \{x_i \in \mathcal{S} \colon x_i \in I_j^{\alpha}\}} (f) + \frac{\varepsilon}{2^{20|j|}},$ then add finitely many points to the partition to make the reverse inequality hold (here, $\V_{\{x_i \in \mathcal{S} \colon x_i \in A\}}(g)$ stands for 
 the variation along the finite partition composed solely by elements in the set $A$). 
 \end{enumerate}
It is easy to see that, if we denote by $\mathcal{S}'$ the partition obtained by the prescribed procedure above, then, as $\V(f) = \V(M^{\alpha}f)$ and $f = M^{\alpha}f$ on $\R \backslash E_{\alpha},$
\[
|\V_{\mathcal{S}'\cap E_{\alpha}}(f) - \V_{\mathcal{S}'\cap E_{\alpha}}(M^{\alpha}f)| \le 2\varepsilon,
\]
which then implies that, by the considerations above, 
\begin{align}
\sum_{j \in \Z} \V_{I_j^{\alpha}}(f) -  \frac{\varepsilon}{4} & \le \sum_{j \in A}\V_{I_j^{\alpha}}(f) \cr 
							      & \le \sum_{j \in A}\V_{ \{x_i \in \mathcal{S}' \colon x_i \in I_j^{\alpha}\}}(f) + \varepsilon\cr
							      & \le \sum_{j \in A}\V_{ \{x_i \in \mathcal{S}' \colon x_i \in I_j^{\alpha}\}}(M^{\alpha}f) + 3\varepsilon \cr 
							      & \le \sum_{j \in \Z}\V_{I_j^{\alpha}}(M^{\alpha}f) + 3 \varepsilon.\cr
\end{align}
As $\varepsilon$ was arbitrary, comparing the first and last terms above and looking back to our proof that in each of the $I_j^{\alpha}$ the variation of $f$ controls that of the maximal function, we conclude that, for each $j \in \Z,$ 
\begin{equation}\label{localpart}
\V_{I_j^{\alpha}}(f) = \V_{I_j^{\alpha}}(M^{\alpha}f).
\end{equation}
This finishes the proof of this claim. 
\end{proof} 

\begin{claim}\label{monot1} Let $f, I_j^{\alpha}$ as above. Then $f$ and $M^{\alpha}f$ are monotone in $I_j^{\alpha}$. 
\end{claim}

\begin{proof} Suppose first that $M^{\alpha}f$ is \emph{not} monotone there. Then it must be V shaped on $I_j^{\alpha}$, and then, by Claim \ref{inteq}, we see that the only possibility for that to happen is if $M^{\alpha}f(c_j) = f(c_j), \,c_j \in I_j^{\alpha}.$ 
This is clearly not possible by the definition of $I_j^{\alpha},$ and we reach a contradition. \\ 

Suppose now that $f$ is \emph{not} monotone over $I_j^{\alpha}.$ As $\V_{I_j^{\alpha}}(f) = \V_{I_j^{\alpha}}(M^{\alpha}f)$ by Claim \ref{inteq}, and $\V_{I_j^{\alpha}}(M^{\alpha}f) = f(r_j) - f(l_j),$ then it is easy to see that, 
no matter what configuration of non-monotonicity we have, it yields a contradiction with the equality for the variations over the interval $I_j^{\alpha}.$ We skip the details, for they are routinely verified.
\end{proof}

\begin{rmk}\label{ceta} Note that this last claim proves also that, if $I_j^{\alpha}$ is \emph{bounded}, $f$ is \emph{non-decreasing} over it and $l_j$ is its left endpoint, then $f(l_j-) \le f(l_j+),$ as 
otherwise we would arrive at a contradiction with the fact that $\V_{I_j^{\alpha}}(f) = \V_{I_j^{\alpha}}(M^{\alpha}f).$ An analogous statement holds for the right endpoint, and analogous conclusions
if $f$ is \emph{non-increasing} instead of non-decreasing over the interval. \\ 
\end{rmk}
Next, we suppose without loss of generality that the function $f$ is non-decreasing on $I_j^{\alpha},$ as the other case is completely analogous. 

\begin{claim}\label{infty} Such an $f$ is, in fact, non-decreasing on $(-\infty,r(I_j^{\alpha})].$ 
\end{claim} 

\begin{proof}Our proof of this fact will go by contradiction: \\ 

First, let $a_j = \inf\{t \in \R; f \text{ is non-decreasing in } [t,r(I_j^{\alpha})]\},$ and define $b_j<a_j$ such that the minimum of $f$ in $[b_j,r_j]$ happens \emph{inside} $(b_j,r_j).$ Of course, such a minimum need not happen at a point, but it surely does happen at a \emph{lateral limit} of a point. 

\begin{subclaim}\label{sclm1} $M^{\alpha}f(a_j) = f(a_j)$ and $f(a_j-) = f(a_j+).$  
\end{subclaim} 
\begin{proof} If $M^{\alpha}f(a_j)>f(a_j)$, then there exists an open interval $E_{\alpha} \supset J'_j \ni a_j$, and, as we proved before, $f$ must be monotone in such an interval. By the definition of $a_j$, $f$ must be 
non-decreasing, which is a contradiction to the definition of $a_j.$  \\ 
 
Now for the second equality: if it were not true, then $a_j$ would be, again, one of the endpoints of a maximal interval $J_j \subset E_{\alpha}.$ If $a_j$ is the left-endpoint, then it means that 
$f(a_j-) > f(a_j+).$ But this is a contradiction, as $f$ then must be non-decreasing on $J_j$, and therefore we would again have that $\V_{J_j}(f) > \V_{J_j}(M^{\alpha}f).$ Therefore, $a_j$ is the right endpoint, and also $f(a_j-) < f(a_j+).$ 
At the present moment an analysis as in Remark \ref{ceta} is already available, and thus we conclude that $f$ shall be non-decreasing on $J_j$, which is again a contradiction to the definition of $J_j.$  
\end{proof} 
We must prove yet another fact that will help us: 
\begin{subclaim}\label{minm} Let 
\[
\mathcal{D} = \{x \in (b_j,r_j)\colon \min(f(x-),f(x+)) \text{ attains the minimum in } (b_j,r_j)\}.
\] 
Then there exists $d \in \mathcal{D}$ such that $f(d-) = f(d+)$ and $M^{\alpha}f(d) = f(d).$
\end{subclaim} 
\begin{proof} If $a_j \in \mathcal{D},$ then our assertion is proved by Subclaim \ref{sclm1}. If not, then $\mathcal{D} \subset (b_j,a_j).$ In this case, pick any point $d_0$ in this intersection. \\ 

\noindent \textit{Case 1: $f(d_0+) = f(d_0-)$.} In this case, there is nothing left undone if $f(d_0) = M^{\alpha}f(d_0).$ Otherwise, we would have that $M^{\alpha}f(d_0) > f(d_0),$ and then there would be an interval $E_{\alpha} \supset J_0 \ni d_0.$ By the fact that \emph{all} the points in $\mathcal{D}$
must lie in $(b_j,a_j),$ and that $f$ is monotone on $J_0$, we see automatically that either  $f(b_j) \le f(d_0),$ a contradiction, or the right endpoint of $J_0$ satisfies $f(r(J_0)) \le f(d_0).$ By the definition of $d_0,$ this inequality has to be an equality, and 
also $f$ must be continuous at $r(J_0)$, by the argument of Remark \ref{ceta}. As an endpoint of a maximal interval $J_0 \subset E_{\alpha},$ we have then $M^{\alpha}f(r(J_0)) = f(r(J_0)).$  \\ 

\noindent \textit{Case 2: $f(d_0+) > f(d_0-)$.} It is easy to see that, in this case, there is an open interval $J \subset E_{\alpha}$ such that either $J \ni d_0$ or $d_0$ is its right endpoint. In either case, we see that $f$ must be non-decreasing over this interval 
$J$, and let again $l_0$ be its left endpoint. As we know, $l_0 \in \mathcal{D}$ again, $l_0 \in (b_j,r_j)$ and, by  Remark \ref{ceta}, we must have that $f(l_0-) = f(l_0+).$ Of course, by being the endpoint we have automatically again that $M^{\alpha}f(l_0) = f(l_0).$ This concludes again this case, and therefore the proof of the subclaim. 
\end{proof} 
The concluding argument for the proof of the Claim \ref{infty} goes as follows: let $d$ be the point from Subclaim \ref{minm}. Then we must have that
\[
f(d) = M^{\alpha}f(d) \ge Mf(d) \ge \dashint_{d-\delta}^{d+\delta} f.
\]
For small $\delta$, it is easy to get a contradiction from that. Indeed, by the properties of the interval $(b_j,r_j]$ one can ensure that it is only needed to analyze 
$\delta \le |d-b_j|.$ The details are omitted. \\ 

This contradiction came from the fact that we supposed that $a_j > -\infty,$ and our claim is established. 
\end{proof} 

Now we finish the proof: If $M^{\alpha}f \le f$ always, we get to the case of a \emph{superharmonic function}, i.e., a function which satisfies $\dashint_{x-r}^{x+r} f(s) \mmd s \le f(x)$ for all $r > 0.$ That is going to be handled in a while. If not, then we analyze the detachment set: 
\begin{enumerate} 
 \item If all intervals in the detachment set are \emph{of one single type}, that is, either all \emph{non-increasing} or all \emph{non-decreasing}, our function must then admit a point $x_0$ such that $f$ is either non-decreasing on $(-\infty,x_0]$ (resp. non-decreasing on $[x_0,+\infty),$) and $f = M^{\alpha}f$ on $(x_0,+\infty)$ (resp. on $(-\infty,x_0)$).
 \item If there is at least one interval of each type, then we must have an interval $[R,S]$ such that 
 \begin{itemize}
  \item $f$ is non-decreasing on $(-\infty,R]$; 
  \item $f$ is non-increasing on $[S,+\infty)$; 
  \item $f = M^{\alpha}f$ on $(R,S).$ 
 \end{itemize}
\end{enumerate}

The analysis is then easily completed for every one of the cases above: If $f = M^{\alpha}f$ over an interval, then, as $M^{\alpha}f \ge Mf$, we conclude that $f$ must be superharmonic there, where by ``locally subharmonic" we mean a function that satisfies $f(x) \ge \dashint_{x-r}^{x+r}f(s) \mmd s$ for all $0\le s \ll_{x} 1$. As superharmonic in one dimension coincides with concave, and concave functions have \emph{at most} one global maximum, then the first case above gives that $f$ is either monotone or has exaclty one point $x_1$ 
such that it is exactly non-decreasing until a point $x_1,$ non-increasing after. The case of monotone functions is easily ruled out, as if $\lim_{x \to \infty} f = L, \, \lim_{x \to -\infty} f = M \Rightarrow \V(f) = |M-L|, \V(M^{\alpha}f) \le \frac{|M-L|}{2}.$ The second case is treated in the exact same fashion, and the result is the same: in the end, the only 
possible extremizers for this problem are functions $f$ such that there is a point $x_1$ such that $f$ is non-decreasing on $(-\infty,x_1),$ and $f$ is non-increasing on $(x_1,+\infty).$ The theorem is then complete. 

\subsection{Proof of Theorem \ref{counterex}}

We start our discussion by pointing out that the measure $\mmd \mu = \delta_0 + \delta_1$ satisfies our Theorem. 

\begin{prop} Let $0 \le \alpha < \frac{1}{3}.$ Then 
$$+\infty = M^{\alpha}\mu(0) > M^{\alpha}\mu\left(\frac{1}{3}\right) < M^{\alpha}\mu\left(\frac{1}{2}\right) > M^{\alpha}\mu\left(\frac{2}{3}\right).$$
That is, $M^{\alpha}\mu$ has a nontrivial local maximum. 
\end{prop}

\begin{proof} By the symmetries of our measure, $M^{\alpha}\mu \left(\frac{1}{3}\right) = M^{\alpha}\mu\left(\frac{2}{3}\right).$ A simple calculation then shows that $M^{\alpha}f\left(\frac{1}{3}\right) = \frac{3(\alpha+1)}{2},$ if $\alpha < \frac{1}{3}.$ 
As $M^{\alpha}\mu \left(\frac{1}{2}\right) \ge M\mu\left(\frac{1}{2}\right) = 2 > \frac{3 \alpha + 3}{2} \iff \alpha < \frac{1}{3},$ we are done with the proof of this proposition. 
\end{proof}

Before proving our Theorem, we mention that our choice of $\frac{1}{3},\frac{1}{2},\frac{2}{3}$ was not random: $\frac{1}{2}$ is actually a \emph{local maximum} of $M^{\alpha}\mu$, while $\frac{1}{3},\frac{2}{3}$ are \emph{local minima}.

\begin{proof}[Proof of Theorem \ref{counterex}] Let $f_n(x) = n (\chi_{[0,\frac{1}{n}]} + \chi_{[1-\frac{1}{n},1]}).$ It is easy to see that $\int g f_n \mmd x \to \int g \mmd \mu(x),$ for each $g \in L^{\infty}(\R)$ that is continuous on $[0,t_0) \cup (t_1,1],$ for some $t_0 < t_1.$ \\

We prove that $M^{\alpha}f_n (x) \to M^{\alpha}\mu (x), \;\forall x \in [0,1].$ This is clearly enough to conclude our Theorem, as then, if we fix $\alpha < \frac{1}{3},$ there will be $n(\alpha) > 0$ such that, for $N \ge n(\alpha),$ 
$$0=f_N\left(\frac{1}{3}\right) <M^{\alpha}f_N \left(\frac{1}{3}\right) < M^{\alpha}f_N\left(\frac{1}{2}\right) > M^{\alpha}f_N\left(\frac{2}{3}\right) > f_N\left(\frac{2}{3}\right) = 0.$$ 
To prove convergence, we argue in two steps. \\ 

The first step is to prove that $\liminf_{n \to +\infty} M^{\alpha}f_n (x) \ge M^{\alpha}\mu (x).$ It clearly holds for $x \in \{0,1\}.$ For $x \in (0,1),$ we see that 
$$M^{\alpha}f_n (x) = \sup_{|x-y|\le \alpha t \le 3 \alpha} \frac{1}{2t}\int_{y-t}^{y+t} f_n(s) \mmd s.$$
But then 
\begin{align*}
M^{\alpha}\mu (x) &=\sup_{|x-y|\le \alpha t \le 3 \alpha} \frac{1}{2t}\int_{y-t}^{y+t}\mmd \mu (s) \cr 
& = \sup_{|x-y|\le \alpha t \le 3 \alpha; t\ge\delta(x)>0} \lim_{n\to \infty} \frac{1}{2t}\int_{y-t}^{y+t} f_n(s) \mmd s \cr 
& \le \liminf_{n \to \infty} M^{\alpha}f_n(x),
\end{align*}
where $\delta(x)>0$ is a multiple of the minimum of the distances of $x$ to either 1 or 0. This completes this part. \\

The second step is to establish that, for every $\varepsilon > 0,$  $\; (1+\varepsilon)M^{\alpha}\mu (x) \ge \limsup_{N \to \infty} M^{\alpha}f_N(x).$  This readily implies the result. \\

To do so, notice that, as $1>x >0,$ then for $N$ sufficiently large, the average that realizes the supremum on the definition of $M^{\alpha}$ has a positive radius bounded bellow and above in $N$. Specifically, we have that
$$M^{\alpha}f_N(x) = \dashint_{y_N-t_N}^{y_N+t_N} f_N(s) \mmd s, \;\; \Delta(x) \ge t_N \ge \delta(x) >0.$$ 
This shows also that $\{y_N\}$ and $\{t_N\}$ must be bounded sequences. Therefore, using compactness, 
\begin{align*}
\limsup_{N \to \infty} M^{\alpha}f_N(x)& = \limsup_{N \to \infty} \dashint_{y_N-t_N}^{y_N+t_N} f(s) \mmd s \cr 
			               & = \lim_{k \to \infty} \dashint_{y_{N_k} - t_{N_k}}^{y_{N_k} + t_{N_k}} f(s) \mmd s \cr
				       & \le (1+\eta)\frac{1}{2t} \limsup_{N \to \infty} \int_{y-(1+\varepsilon/2) t}^{y+(1+\varepsilon/2) t} f_N(s) \mmd s \cr
				       & = (1+\eta)(1+\varepsilon/2) \dashint_{y-(1+\varepsilon/2)t}^{y +(1+\varepsilon/2)t} \mmd \mu (s) \cr 
& \le (1+\varepsilon) M^{\alpha}\mu(x), \cr
\end{align*}
where we assume that the sequence $\{n_k\}$ is suitably chosen so that the convergence requirements all hold. 
If we make $N$ sufficiently large, and take $\eta$ depending on $\varepsilon$ such that $(1+\eta)(1+\varepsilon/2) < 1+\varepsilon,$ we are done with the second part.  
\end{proof}

\section{Proof of Theorems \ref{lip} and \ref{lipcont}} The idea for this proof is basically the same as before: analyze local maxima in the detachment set in this Lipschitz case, proving that the maximal function is either V shaped or monotone in its composing intervals, 
\emph{if} the Lipschitz constant into consideration is less than  $\frac{1}{2}$. The endpoint case is done by approximation, and we comment on how to do it later. By the end, we sketch on how to build the mentioned counterexamples. 

\subsection{Analysis of maxima of $M^1_N$ for $\text{Lip}(N)< \frac{1}{2}$} Let first $(a,b)$ be an interval on the real line, such that there exists a point $x_0$, maximum of $M^1_Nf$ over $(a,b)$, with the property that 
\[
M^1_Nf(x_0) > \max\{M^1_Nf(a),M^1_Nf(b)\}.
\]
Therefore, we wish to prove that, for some point in $(a,b)$, $M^1_Nf = f.$ We begin with the general strategy: let us suppose that this is not the case. 
Then there must be an average $u(y,t) =\frac{1}{2t} \int_{y-t}^{y+t}|f(s)|\mmd s$ with $N(x_0) \ge t>0$, $|x_0-y| \le t$ and $M^1_Nf(x_0) = u(y,t).$ \\

Now we want to find a neighbourhood of $x_0$ such that there is $R=R(x_0)>0$ such that, for all $x \in I,$ $M^1_{\equiv R}f(x) = M^1_Nf(x_0).$ \\ 

By Lemma \ref{BPL}, we can suppose that either $y = x_0 - t$ or $y=x_0 + t,$ as we can show that $y \in (a,b).$ Without loss of generality, let us assume that $y = x_0-t.$ \\

\noindent\textit{Case (a): $t<N(x_0).$} This is the easiest case, and we rule it out with a simple observation: let $I$ be an interval for which $x_0$ is an endpoint and
such that, for all $x \in I,\, N(x)> t.$ We claim then that, for $x \in I,$ $M^1_{\equiv t+\varepsilon}f(x) = M^1_Nf(x_0),$ if $\varepsilon$ is sufficiently small. Indeed, if $\varepsilon$ 
is sufficiently small, then $M^1_{\equiv t+\varepsilon}f(x) \le M^1_Nf(x) (\le M^1_Nf(x_0))$ for every $x \in I.$ But then we see also that $(x_0-t,t)$ belongs to the region $\{y:|x-y|\le s \le N(x)\},$ as 
then $|(x_0-t)-x|=x+t-x_0 \le t < t+\varepsilon <N(x).$ This shows that 
\[ 
M^1_Nf(x_0) \le \inf_{x \in I} M^1_{\equiv t+\varepsilon}f(x) \le \sup_{x \in I} M^1_{\equiv t+\varepsilon}f(x) \le M^1_Nf(x_0).
\]
As before, we finish this case with \cite[Lemma~3.6]{aldazperezlazaro}, as then it guarantees us that $M^1_{\equiv t+\varepsilon}f(x) = f(x)$ for every point in this interval $I$. \\

\noindent\textit{Case (b): $t=N(x_0).$} In this case, we have to use Lemma \ref{square}. Namely, we wish to include the point $(x_0 - N(x_0),N(x_0))$ in the region 
\[
\{(z,s):|z-x| + |s-N(x)| \le N(x)\},
\]
for $x_0 - \delta < x<x_0$, $\delta$ sufficiently small. 


Let then $\varepsilon >0$ and $x$ close to $x_0$ be such that $N(x) \ge N(x_0) - \varepsilon.$ We have already a comparison of the form
\[
M^1_Nf(x) \ge M^1_{\equiv N(x_0) - \varepsilon}f(x).
\]
We want to conclude that there is an interval $I$ such that $M^1_{\equiv N(x_0)-\varepsilon}f$ is constant on $I$. We want then the point $(x_0 - N(x_0),N(x_0))$ to lie on the set 
\[ 
\{(z,s): |z-x| + |s-N(x_0) + \varepsilon| \le N(x_0) - \varepsilon\}.
\]
But this is equivalent to 
\[
x-x_0 + N(x_0) + \varepsilon \le N(x_0) - \varepsilon \iff |x-x_0| \ge 2\varepsilon.
\]
So, we can only afford to to this if $x$ is somewhat not too close to $x_0.$ But, as $\text{Lip}(N) < \frac{1}{2}$ in this case, we see that 
\[ 
|N(x)-N(x_0)| \le \text{Lip}(N)|x-x_0| \Rightarrow N(x) \ge N(x_0) -\text{Lip}(N)|x-x_0| > N(x_0) - \varepsilon \iff 
\]
\[
|x-x_0| \le \frac{1}{\text{Lip}(N)} \varepsilon.
\]
Therefore, we conclude that, on the non-trivial set 
\[
\{x \in \R: \frac{1}{\text{Lip}(N)} \varepsilon \ge |x-x_0| \ge 2\varepsilon\},
\]
it holds that $M^1_Nf(x_0) \ge M^1_Nf(x) \ge M^1_{\equiv N(x_0) - \varepsilon}f(x) \ge M^1_Nf(x_0) \ge M^1_Nf(x).$ By \cite[Lemma~3.6]{aldazperezlazaro}, $M^1_{\equiv N(x_0)-\varepsilon}f(x) = M^1_Nf(x) = f(x).$ 
This concludes the 
analysis in this case, and also finishes this part of the section, as the finishing argument here is then the same as the one used in Theorem \ref{angle}, and we therefore omit it.
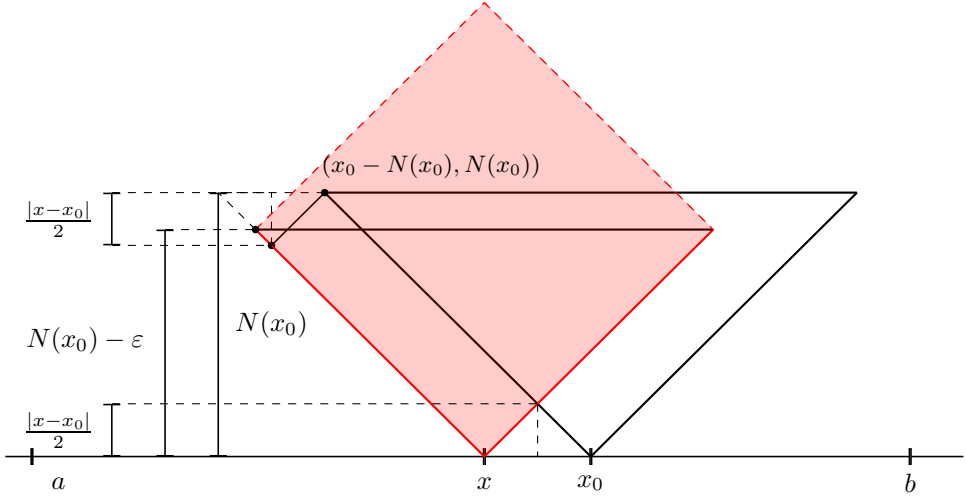
\begin{figure}
\centering
\begin{tikzpicture}[scale=0.7]
\draw[-|,semithick] (-8,0)--(-7.5,0); 
\draw[|-|,semithick] (-7.5,0)--(1,0);
\draw[|-|,semithick] (1,0)--(3,0);
\draw[|-|,semithick] (3,0)--(9,0); 
\draw[|-,semithick] (9,0)--(10,0);
\draw[-,thick] (3,0)--(8,5);
\draw[-,thick] (3,0)--(-2,5);
\draw[-,thick] (-2,5)--(8,5); 
\draw[-,thick, color=red] (1,0)--(5.3,4.3);
\draw[-,thick, color=red] (1,0)--(-3.3,4.3); 
\draw[-,thick] (-3.3,4.3)--(5.3,4.3);
\draw[|-|,semithick] (-5,4.3)--(-5,0);
\draw (-6.5,2.2) node {$N(x_0)-\varepsilon$};
\draw[dashed] (-5,4.3)--(-3.3,4.3);
\draw[|-|,semithick] (-4,5)--(-4,0);
\draw (-3,2.5) node {$N(x_0)$};
\draw[dashed] (2,1)--(2,0);
\draw[-,semithick] (-3,4)--(-2,5);
\draw[dashed] (-3,4)--(-3,5);
\draw[|-|,semithick] (-6,4)--(-6,5);
\draw (-7,4.5) node {$\frac{|x-x_0|}{2}$};
\draw (-3,4) node[circle,fill,inner sep=1pt]{} ;
\draw[dashed] (-6,4)--(-3,4);
\draw[dashed] (-6,5)--(-3,5);
\draw[dashed] (-2,5)--(-4,5); 
\draw[dashed] (-3.3,4.3)--(-4,5);
\draw[dashed, color=red] (-3.3,4.3)--(1,8.6); 
\draw[dashed, color=red] (1,8.6)--(5.3,4.3);
\draw[dashed, draw=red, fill=red, fill opacity=0.2] (1,0) -- (5.3,4.3) -- (1,8.6) -- (-3.3,4.3) -- (1,0);
\draw (3,-0.5) node {$x_0$}; 
\draw (1,-0.5) node {$x$};
\draw (-7,-0.5) node {$a$};
\draw (9,-0.5) node {$b$};
\draw (0,5.5) node {{\small$(x_0 - N(x_0),N(x_0))$}};
\draw[|-|,semithick] (-6,0) -- (-6,1);
\draw (-7,0.5) node {$\frac{|x-x_0|}{2}$}; 
\draw[dashed] (-6,1)--(2,1);
\draw (-2,5) node[circle,fill,inner sep=1pt]{} ;
\draw (-3.3,4.3) node[circle,fill,inner sep=1pt]{};
\end{tikzpicture}
\caption{Illustration of proof of case (b).} 
\end{figure} 
\vspace{1cm}

\subsection{The critical case $\text{Lip}(N)=\frac{1}{2}$} The argument is pretty simple: we build explicitly a suitable sequence of approximations of $N$ such that they all have Lipschitz constants less than $\frac{1}{2}.$ By our already proved results, 
this will give us the result also in this case. \\ 

Explicitly, let $N$ be such that $\text{Lip}(N) = \frac{1}{2}$ and $f \in BV(\R)$. Let then $\mathcal{P} = \{x_1 < \cdots < x_M\}$ be any partition of the real line. Let $J \gg 1$  be a large integer, and divide the interval $[x_1,x_M]$ into $J$ equal parts, that we call $(a_j,b_j)$. Define also the numbers
\[ 
\Delta_j = \frac{N(b_j)-N(a_j)}{b_j-a_j}.
\]
We know, by hypothesis, that $\Delta_j \in [-1/2,1/2].$ Let then $\tilde{\Delta}_j = \frac{1}{2} - \frac{1}{J^3},$ and define the function 
\[
\tilde{N}(x) = \begin{cases} 
                N(x_1), & \text{ if } x \le x_1, \cr
                N(x_1) + \tilde{\Delta}_1(x-x_1), & \text{ if } x \in (a_1,b_1),\cr
                \tilde{N}(b_{j-1}) + \tilde{\Delta}_j (x-b_{j-1}), & \text{ if } x \in (a_j,b_j), \cr
                \tilde{N}(b_J), & \text{ if } x \ge x_N.\cr
               \end{cases}
\]
It is obvious that this function is continuous and Lipschitz with constant $\frac{1}{2} - \frac{1}{J^3}.$ If $x \in [x_1,x_N],$ then 
\[ 
|\tilde{N}(x) - N(x)| = |\tilde{N}(x) - \tilde{N}(x_1) + N(x_1) - N(x)| \le \int_{x_1}^x |\frac{1}{2} - \frac{1}{J^3} - \frac{1}{2}| \mmd t \le \frac{|x_M - x_1|}{J^3}.
\]
We now choose $J$ such that the right hand side above is less than $\delta > 0,$ which is going to be chosen as follows: for the same partition $\mathcal{P},$ we let $\delta > 0$ be such that
\[ 
|\tilde{N}(x_i) - N(x_i)|< \delta \Rightarrow |M^1_{N(x_j)}f(x_j) - M^1_{\tilde{N}(x_j)}f(x_j)| < \frac{\varepsilon}{2M}.
\]
This can, by continuity, always be accomplished. This implies that, using the previous case, 
\[
\V_{\mathcal{P}}(M^1_N f) \le \V_{\mathcal{P}}(M^1_{\tilde{N}}f) + \varepsilon \le \V(M^1_{\tilde{N}}f) + \varepsilon \le \V(f) + \varepsilon.
\]
Taking the supremum over all possible partitions and then taking $\varepsilon \to 0$ finishes also this case, and thus the proof of Theorem \ref{lip}.

\subsection{Counterexample for $\text{Lip}(N) > \frac{1}{2}$}
Finally, we build examples of functions with $\text{Lip}(N) > \frac{1}{2}$ and $f \in BV(\R)$ such that 
\[
\V(M_Nf) = + \infty. 
\]
Fix then $\beta > \frac{1}{2}$ and let a function $N$ with $\text{Lip}(N) = \beta$ be defined as follows: 
\begin{enumerate}
 \item First, let $x_0 = \frac{2}{2\beta+1}.$ Let then $N(0) = 1,\,N(x_0) = \frac{x_0}{2}$ and extend it linearly in $(0,x_0).$ 
 \item Let $x'_K$ be the solution to the equation $\beta x - \beta x_{K-1} + \frac{x_{K-1}}{2} = \frac{x+1}{2}  \iff x'_K = x_{K-1} + \frac{1}{\beta - \frac{1}{2}}.$ 
 \item At last, take $x_K = x'_K + \frac{1}{2\beta +1},$ and define for all $K \ge 1$ $N(x_K) = \frac{x_K}{2},\, N(x'_K) = \frac{x'_K+1}{2} ,$ extending it linearly on $(x_{K-1},x'_K)$ and $(x'_K,x_K).$ 
\end{enumerate}
As $\{x'_K\}_{K \ge 1}$ is an arithmetic progression, we see that 
\[ 
\sum_{K \ge 0} \frac{1}{x'_K} = +\infty.
\]
Moreover, define $f(x) = \chi_{(-1,0)}(x).$ We will show that, for this $N$, we have that 
\[ 
\V(M^1_Nf) = +\infty.
\]
In fact, it is not difficult to see that: 
\begin{enumerate}
 \item \textit{$M^1_Nf(x_K) =0, \, \forall K \ge 0.$} This is due to the fact that the maximal intervals $(y-t,y+t)$ that satisfy $|x_K - y| \le t \le N(x_K)$ are still contained in $[0,+\infty),$ which is of course disjoint from $(-1,0).$  
 \item \textit{$M^1_Nf(x'_K) \ge \frac{1}{x'_K+1}$.} This follows from 
 \[ 
 M^1_Nf(x'_K) \ge \frac{1}{2Nf(x'_K)} \int_{-1}^{x'_K} f(t) \mmd t = \frac{1}{x'_K + 1}.
 \]
\end{enumerate}
This shows that 
\[
\V(M^1_Nf) \ge \sum_{K=0}^{\infty} |M^1_Nf(x'_K) - M^1_Nf(x_K)| = \sum_{K=0}^{\infty} \frac{1}{x'_K + 1} = + \infty.
\]
This construction therefore proves Theorem \ref{lipcont}. 
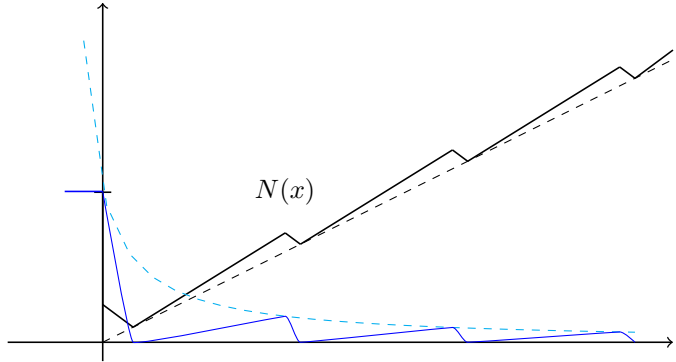
\begin{figure} 
\centering
\begin{tikzpicture}[scale=0.5]
\draw[->,semithick] (0,-0.5) -- (0,9);
\draw[->,semithick] (-2.5,0) -- (15,0); 
\draw[-|,semithick] (0,0) -- (0,4);
\draw[-,semithick] (0,1) -- (4/5,2/5);
\draw[-,semithick] (4/5,2/5) -- (24/5,29/10);
\draw[-,semithick] (24/5,29/10) -- (26/5,13/5);
\draw[-,semithick] (26/5,13/5) -- (46/5, 51/10);
\draw[-,semithick] (46/5, 51/10) -- (48/5, 24/5);
\draw[-,semithick] (48/5,24/5) -- (68/5, 73/10); 
\draw[-,semithick] (68/5,73/10) -- (70/5, 7);
\draw[-,semithick] (14,7) -- (15, 45/4 + 7 - 42/4); 
\draw[dashed, cyan, domain=-0.5:14] plot(\x,{4/(1+\x)});
\draw[dashed] (0,0) -- (15,15/2); 
\draw (24/5,4) node{$N(x)$};
\draw[-,semithick, draw=blue] (-1,4)--(0,4);
\draw[draw=blue, fill=none] plot [smooth, tension=0.1] coordinates {(0,4) (4/5,0) (24/5,20/29) (26/5,0) (46/5,20/51) (48/5,0) (68/5,20/73) (14,0)}; 
\end{tikzpicture}
\caption{A counterexample in the case of $\text{Lip}(N) = \frac{3}{4}.$ The dashed lines are the graphs of $\frac{x}{2}$ and $\frac{1}{1+x}$, and the non-dashes ones the graphs of $M^1_Nf$ and $N$ in this case.}
\end{figure}

\section{Comments and remarks} 

\subsection{Nontangential maximal functions and classical results} Here, we investigated mostly the regularity aspect of our family $M^{\alpha}$ of nontangential maximal functions, and looked for the sharp constants 
in such bounded variation inequalities. One can, however, still ask about the most classical aspect studied by Melas \cite{Melas}: Let $C_{\alpha}$ be the least constant such that we have the following inequality: 
\[ 
|\{ x \in \R\colon M^{\alpha}f(x) > \lambda \}| \le \frac{C_{\alpha}}{\lambda} \|f\|_1. 
\] 
By \cite{Melas}, we have that, for when $\alpha = 0$, then $C_0 = \frac{11 + \sqrt{61}}{12},$ and the classical argument of Riesz \cite{riesz} that $C_1 = 2.$ Therefore, $\frac{11 + \sqrt{61}}{12} \le C_{\alpha} \le 2, \,\,\forall \alpha \in (0,1).$ 
Nevertheless, the exact values of those constants is, as long as the author knows, still unknown. \\
\subsection{Bounded variation results for mixed Lipschitz and nontangential maximal functions} In Theorems \ref{lip} and \ref{lipcont}, we proved that, for the \emph{uncentered} Lipschitz maximal function $M_N$, we have sharp 
bounded variation results for $\text{Lip}(N) \le \frac{1}{2},$ and, if $\text{Lip}(N)>\frac{1}{2}$, we cannot even assure \emph{any} sort of bounded variation result. \\ 

We can ask yet another question: if we define the \emph{nontangential Lipschitz maximal function} 
\[ 
M^{\alpha}_Nf(x) = \sup_{|x-y| \le \alpha t \le \alpha N(x)} \frac{1}{2t} \int_{y-t}^{y+t} |f(s)| \mmd s,
\]
then what should be the best constant $L(\alpha)$ such that, for $\text{Lip}(N) \le L(\alpha),$ then we have some sort of bounded variation result like $\V(M^{\alpha}_Nf) \le A\V(f),$ and, for each $\beta > L(\alpha),$ there 
exists a function $N_{\beta}$ and a function $f_{\beta} \in BV(\R)$ such that $\text{Lip}(N_{\beta}) = \beta$ and $\V(M_{N_{\beta}}f_{\beta}) = +\infty?$ Regarding this question, we cannot state any kind of 
sharp constant bounded variation result, but the following is still attainable: it is possible to show that the first two lemmas of O. Kurka \cite{kurka} are adaptable in this context \emph{if} we suppose that 
\[ 
\text{Lip}(N) \le \frac{1}{\alpha + 1},
\]
and then we obtain our results, with a constant that is even independent of $\alpha \in (0,1).$ On the other hand, our example used above in the proof of Theorem \ref{lipcont} is easily adaptable as well, and therefore one might prove
the following Theorem: 
\begin{theorem} Let $\alpha \in [0,1]$ and $N$ be a Lipschitz function with $\text{Lip}(N) \le \frac{1}{\alpha + 1}.$ Then, for every $f \in BV(\R),$ we have that 
\[
\V(M^{\alpha}_Nf) \le C \V(f),
\]
where $C$ is independent of $N,f,\alpha.$ Moreover, for all $\beta > \frac{1}{\alpha + 1},$ there is a function $N_{\beta}$ and 
\[ 
f(x) = \begin{cases} 
        1, & \text{ if } x \in (-1,0);\cr
        0, & otherwise, \cr
       \end{cases}
\]
with $\text{Lip}(N_{\beta}) = \beta$ and $\V(M^{\alpha}_{N_{\beta}}f) = + \infty.$ 
\end{theorem}

\subsection{Increasing property of maximal $BV-$norms}\label{commrem} Theorem \ref{angle} proves that, if we define 
\[
B(\alpha) := \sup_{f \in BV(\R)} \frac{\V(M^{\alpha}f)}{\V(f)},
\]
then $B(\alpha) = 1$ for all $\alpha \in [\frac{1}{3},1].$ We can, however, with the same technique, show that $B(\alpha)$ is increasing in $\alpha > 0,$ and also that $B(\alpha) \equiv 1 \,\, \forall \alpha \in [\frac{1}{3},+\infty).$ 
Indeed, we show that, for $f \in BV(\R)$ and $\beta > \alpha,$ then $\V(M^{\alpha}f) \ge \V(M^{\beta}f).$ The argument uses the maximal attachment property in the following way: let, as usual, $(a,b)$ be an interval where $M^{\beta}f$ has 
a local maximum \emph{inside} it, at, say, $x_0$, and 
\[
M^{\beta}f(x_0) > \max(M^{\beta}f(a),M^{\beta}f(b)).
\]
Then, as we have that $M^{\beta}f \ge M^{\alpha}f$ \emph{everywhere}, we have two options: 
\begin{itemize} 
\item \textit{If $M^{\beta}f(x_0) =f(x_0),$} we do not have absolutely anything to do, as then also $M^{\alpha}f(x_0) = M^{\beta}f(x_0).$ 
\item \textit{If $M^{\beta}f(x_0) = u(y,t),$ for $t > 0$,} we have -- as in the proof of Theorem \ref{angle} -- that $(y-\beta t, y+\beta t) \subset (a,b).$ But it is then obvious that 
\[
M^{\alpha}f(y) \ge u(y,t) = M^{\beta}f(x_0) \ge M^{\beta}f(y) \ge M^{\alpha}f(y).
\]
\end{itemize}
Therefore, we have obtained a form of the maximal attachment property, and therefore we can apply the standard techniques that have been used through the paper to this case, and it is going to yield our result. \\ 

This shows directly that $B(\alpha) \le 1, \forall \alpha \ge 1,$ but taking $f(x) = \chi_{(0,1)}$ as we did several times shows that actually $B(\alpha)=1$ in this range.

\end{document}